\documentclass[12pt]{amsart}

\usepackage[all]{xy}
\usepackage{graphics,bm}
\usepackage{graphicx}
\usepackage{color}
\usepackage[T1]{fontenc}
\usepackage{parskip}
\usepackage{enumitem}
\usepackage{comment}
\usepackage{amsfonts}
\usepackage{mathrsfs}
\usepackage{amsmath,amssymb, mathtools}
\usepackage[shortcuts]{extdash}
\usepackage{dsfont}
\usepackage{fullpage}
\usepackage[colorlinks]{hyperref}
\usepackage{bbm}
\usepackage{cleveref}
\usepackage{marginnote}

\newtheorem{theorem}{Theorem}[section]
\newtheorem*{theorem*}{Theorem}
\newtheorem{lemma}[theorem]{Lemma}

\newtheorem{corollary}[theorem]{Corollary}

\newtheorem*{question*}{Question}

\theoremstyle{definition}
\newtheorem{example}[theorem]{Example}

\newtheorem{remark}[theorem]{Remark}

\newtheorem*{remark*}{Remark}
\newtheorem*{remarks*}{Remarks}
\newtheorem*{corollary*}{Corollary}

\numberwithin{figure}{section}
\numberwithin{table}{section}
\numberwithin{equation}{section}

\newcommand{\la}{\langle}
\newcommand{\ra}{\rangle}

\newcommand{\OP}[1]{\operatorname{#1}}
\newcommand{\Aut}{\operatorname{Aut}}
\newcommand{\SAut}{\operatorname{SAut}}
\newcommand{\Epi}{\operatorname{Epi}}

\newcommand{\F}{\operatorname{F}}
\newcommand{\A}{\operatorname{A}}
\newcommand{\Adj}{\operatorname{Adj}}
\newcommand{\Sq}{\operatorname{Sq}}
\newcommand{\Opp}{\operatorname{Opp}}
\newcommand{\E}{\operatorname{E}}
\newcommand{\SL}{\operatorname{SL}}

\newcommand{\myast}{\hspace{-1pt} \ast \hspace{-1pt}}

\newcommand{\BB}[1]{{\mathbb #1}}
\newcommand{\FP}{\mathbb{Z} \myast \F_k}

\title[Sampling elements of a finite group]{sampling elements of a finite group: efficiency of the product replacement algorithm with an  accumulator.}

\begin{document}

\author{Michał Marcinkowski}
\address{Department of Mathematics, University of Wrocław, Poland}
\email{michal.marcinkowski@math.uni.wroc.pl}
\author{Piotr Mizerka}
\address{Faculty of Mathematics and Computer Science, Adam Mickiewicz Univeristy in Poznań, Poland}
\email{piotr.mizerka@amu.edu.pl}

\begin{abstract}
Let $G$ be a finite group generated by $k$ elements. The well-known product replacement algorithm provides an effective method for sampling generating sets of $G$. We study a refinement of this algorithm that is designed to output individual elements of $G$. We show that after $O(k^2\log|G|)$ steps, the distribution of the output is close to uniform on~$G$, which improves upon the best results known to date. The proof proceeds via spectral gap estimates and uses computer assisted calculations.
\end{abstract}

\maketitle

\section{Introduction}

Let $G$ be a finite group. We are interested in an efficient algorithm that samples elements of $G$ from a nearly uniform distribution. Such algorithms play a central role in computational group theory, as many other algorithms require access to random group elements \cite{10.1145/96877.96918, 10.1112/plms/s3-65.3.555}. Often one deals with finite groups that are astronomically large, e.g., the Rubik's cube group has an order about $10^{19}$. Thus one cannot simply generate the list of elements and choose one at random, as generating such a list would require prohibitive time and memory.  We therefore need some additional information on $G$ to use in the sampling procedure. It is not a restrictive assumption to suppose that, together with $G$, we are given a generating tuple $S$ (in this context it is more convenient to use tuples rather than generating sets). We can then multiply the generators in various ways and thereby produce elements of $G$ as output.   

To this end, we need to assume that we have an efficient procedure for inverting and multiplying elements of $G$, as well as an efficient way to represent the results of these operations on a computer. Such groups are called  ``black-box'' groups, since when describing particular algorithms operating on them, we will not be concerned with the details of how these operations are implemented. Subgroups of permutation groups or matrix groups over finite rings are among the most important examples of black-box groups.

The simplest method to sample an element from $G$ given its generating tuple $S$, is to perform a lazy random walk on the Cayley graph of $(G,S \cup S^{-1})$ starting from the identity. By the standard theory, the distribution of this random walk converges to the uniform distribution. However, even for very simple groups such as $\BB Z_n$, the number of steps required in order to be close to the uniform distribution can make this approach intractable if $n$ is large. 

There exist algorithms that are much more efficient than a random walk on a Cayley graph. One of the most important techniques for designing such samplers is based on the product replacement algorithm. This algorithm was first introduced in \cite{MR1356111} and is again based on a random walk, but this time not on the Cayley graph, but rather on the graph of all generating $k$-tuples, where $k = |S|$. Thus the algorithm first outputs a generating tuple, and then one can take the first element of the tuple as the output element of $G$. It was reported that this technique works well in practice and passes statistical tests. However, in general the distribution of this random element does not converge to the uniform distribution, see \cite{MR2053017} or \Cref{e:Z^k}. For a survey of the product replacement algorithm and related problems see \cite{MR1829489}.

In order to overcome this difficulty, variants of the product replacement algorithm were proposed \cite{MR1929718}. In particular, in the ``rattle'' version it is clear that the limiting distribution is uniform, but the rate of convergence remained unknown.

It was observed that the behavior of the original product replacement algorithm, which samples a generating $k$-tuple, is related to Kazhdan's property (T) for $\SAut(\F_k)$ \cite{MR1815215}. The question whether $\SAut(\F_k)$ has property (T) was a longstanding question and was answered affirmatively in the breakthrough papers \cite{MR4224715, MR4023374} for $k>4$, and in \cite{nitsche2022computerproofspropertyt} the remaining case of $\SAut(\F_4)$ was settled. Even more importantly, very good estimates were obtained for the corresponding spectral gaps. From those estimates it immediately follows that the mixing rate of the original product replacement algorithm is of order $O(k^2\log|G|)$. This means that after $O(k^2\log|G|)$ steps the algorithm outputs a generating tuple from a distribution that is close to uniform.

The problem of finding good estimates for product replacement algorithms that sample an element from $G$ has remained open. A step towards such estimates was made in \cite{MR4709065}, where the ``rattle'' version was analyzed. However, it is assumed there that $k$ is of order $\log|G|$, which can be prohibitively large when $G$ is large. 

In this paper we focus on a variant of the product replacement algorithm that is a slight modification of the ``rattle'' and is easier to analyze from the point of view of group actions. We call it the \textit{product replacement accumulator algorithm}. It is based on a random walk on the space of generating tuples extended by one additional element - the accumulator. After~$t$ steps of the random walk, the accumulator is returned. Our main theorem is the following.

\begin{theorem*}[\Cref{t:rate-of-convergence}]
Suppose that $G$ is a finite group equipped with a generating $k$-tuple, with $k>5$, and let $U_G$ denote the uniform distribution on $G$. Let $\nu_t$ be the distribution of the random element returned by the product replacement accumulator algorithm after $t$ steps. Then

$$
\|\nu_t - U_G\|_{\mathrm{tv}} < \epsilon \text{ for } t \geq \frac{23k^2}{k-5}((k+1)\log|G| + \log(\epsilon^{-1})),
$$

where $\|\!\cdot\!\|_{\mathrm{tv}}$ is the total variation distance and $\log$ is the natural logarithm. 
\end{theorem*}

This result should be interpreted as follows: we need to make about $\frac{23k^2}{k-5}(k\log|G| + 1)$ steps to ensure that $\|\nu_t - U_G\|_{\mathrm{tv}} < e^{-1}$, where $e$ denotes Euler’s number. This is the most expensive part. From then on, every time we make an additional $\frac{23k^2}{k-5}$ steps, we are guaranteed that $\nu_t$ is $e$-times closer to the uniform distribution. 

We want to highlight, that one of the most interesting aspects of \Cref{t:rate-of-convergence}, as well as \cite[Section 5.2]{MR4224715}, is that the bound on $t$ depends only on $|G|$. It does not depend at all on the more intrinsic structure of $G$.

The estimates in \Cref{t:rate-of-convergence} are virtually the same as those for the mixing time of the original product replacement algorithm obtained in \cite{MR4224715}, so passing from sampling a generating tuple to sampling an element of $G$ does not incur much additional overhead. We want to stress that the product replacement accumulator algorithm is very simple: in one step it performs only one multiplication and at most one inversion of group elements. There are no other computations involved, so what matters is only how many steps of the algorithm one needs to perform and the cost of group operations in $G$. Moreover, the required memory is very modest: one needs to store only $k+1$ elements of $G$.

It is possible that the convergence of the product replacement algorithm and its different variants is much faster than the estimates obtained so far. However, better asymptotic bounds cannot be obtained by a classical analysis of the spectral gap, see \Cref{r:sharp}.

Finally, let us very briefly compare the product replacement accumulator algorithm to other sampling methods. In the following asymptotic bounds we assume that $k \leq \log|G|$. The Schreier-Sims algorithm \cite{MR1112272, Seress_2003}, among other applications, can be used to sample an element from a subgroup $G$ of the permutation group $S_n$. It has a costly preprocessing phase requiring in the known implementations at least $\Theta(n^2 \log^3 |G|)$ time, but once this phase is completed, one can sample \textit{uniformly} from $G$ in time $\Theta(\log|G|)$. This algorithm requires storing at least $\Theta(n\log|G|)$ elements of $G$ \cite[Theorem 4.2.4]{Seress_2003}. Thus, even for subgroups of permutation groups, our algorithm may be a better choice if one needs to sample only a few elements from $G$. Another very interesting family of algorithms is based on the random subproduct technique \cite{10.1145/103418.103440}. An efficient version was proposed by Cooperman and its correctness was proved in \cite{Dixon2008}. It requires $\Theta(\log^2 |G|)$ steps in the preprocessing phase, after which one can sample in time $\Theta(\log|G|)$ from an almost uniform distribution. The drawback of this algorithm is that the preprocessing phase may fail (although the probability of failure is very small), and it requires storing $\Theta(\log|G|)$ elements of $G$. In terms of time efficiency, the product replacement accumulator algorithm is guaranteed to be better than the Cooperman algorithm in situations where $k^2 < \log|G|$ and the number of elements to be sampled is small. Additionally, the hidden constants neglected in the asymptotic estimates can potentially be large in the Cooperman algorithm.

\textbf{Outline}. In \Cref{s:pra} we recall the original product replacement algorithm. In \Cref{s:random-walk} we define $\Sigma'_k(G)$, the graph of generating $k$-tuples of $G$ with an accumulator. We then define the group $\A_k$ such that, for any $G$, the graph $\Sigma'_k(G)$ is a Schreier graph of $\A_k$. In \Cref{s:praa} we define the product replacement accumulator algorithm and show that the distribution of the accumulator converges to the uniform distribution on $G$. In \Cref{s:main-result} we prove the main result. There, combining computer-assisted and human-made calculations, we show that $\Delta^2_k - 0.35(K-5)\Delta_k$ is a sum of squares in the group ring $\BB R[\A_k]$. \Cref{t:rate-of-convergence} follows immediately from this observation. For completeness, the details are given in \Cref{ss:convergence}. In \Cref{s:semidirect} we show that $\A_k$ is a semidirect-product and observe that property (T) for $\A_k$ follows trivially from \cite[Proposition 10]{MR4023374}. However, this gives significantly weaker bounds on the spectral gaps. Finally, in \Cref{s:computer-calc} we provide an overview of the computer methods used.

\textbf{Acknowledgments}. We thank Marek Kaluba for his help in implementing the code used for computer calculations. We would like to thank Adam Mickiewicz University for providing access to the cluster infrastructure of the Faculty of Mathematics and Computer Science.

\section{Product replacement algorithm}\label{s:pra}

Let $\F_k = \la x_1, \ldots, x_k \ra$ be the free group of rank $k$. Denote by $\Aut(\F_k)$ the automorphism group of $\F_k$. Note that an automorphism is defined by its images on the generators $x_1, \ldots, x_k$. We shall use the notation $\psi(g_1,\ldots,g_s) = (\psi(g_1),\ldots,\psi(g_s))$ for $\psi \in \Aut(\F_k)$ and $(g_1,\ldots,g_s) \in \F_k^s$.
The group $\Aut(\F_k)$ is finitely generated and contains a subgroup of index~$2$, denoted $\SAut(\F_k)$, which is generated by the following set of Nielsen automorphisms:

$$
R_{ij}^\pm(x_1,\ldots,x_k) = (x_1,\ldots,x_ix_j^\pm,\ldots,x_k)
$$
$$
L_{ij}^\pm(x_1,\ldots,x_k) = (x_1,\ldots,x_j^\pm x_i,\ldots,x_k).
$$

That is, to the $i$-th generator we multiply the $j$-th generator or its inverse from the left or the right.

Let $G$ be a finite group and $k \in \mathbb{N}$.
Denote by $\Epi(\F_k,G)$ the set of all epimorphisms from $\F_k$ to $G$. On $\Epi(\F_k,G)$ we have a left $\SAut(\F_k)$-action by precomposition:
$$\psi.e(x) = e(\psi^{-1}(x))$$ 
for $\psi \in \SAut(\F_k)$, $e \in \Epi(\F_k,G)$ and $x \in \F_k$.
Let \marginnote{$\Gamma_k(G)$} $\Gamma_k(G)$ be the Schreier graph corresponding to this action and the generating set 
$$N_k = \{ R^\pm_{ij}, L^\pm_{ij} \colon 1 \leq i,j \leq k \}.$$
The vertices of $\Gamma_k(G)$ are elements of $\Epi(\F_k,G)$ and there is a directed labeled edge $e_1 \xlongrightarrow{s} e_2$ if $e_2 = s.e_1$ for $s \in N_k$. There can be multiple edges between vertices of $\Gamma_k(G)$ as well as loops. The graph is regular of out-degree and in-degree equal to $4k(k-1)$.

The description above of $\Gamma_k(G)$ is convenient, as it carries a clear action of $\SAut(\F_k)$.
Let us now give a different description, which is more natural from the point of view of the product replacement algorithm. 

Note that each epimorphism $e \in \Gamma_k(G)$ is defined by its values on the generators $e(x_i)$, and that $(e(x_1),\ldots,e(x_k))$ is a generating set of $G$.
The assignment $e \mapsto (e(x_1),\ldots,e(x_k))$ gives a bijective correspondence between the vertices of $\Gamma_k(G)$ and generating $k$-tuples of $G$. Thus

$$
\Epi(\F_k,G) = \{ (g_1,\ldots,g_k) \colon \langle g_1,\ldots,g_k \rangle = G \}.
$$

The elements $g_i$ are called generators. Two generating $k$-tuples are connected if one can move from one to another by multiplying the $j$-th generator or its inverse to the $i$-th generator from the left or right. For example

$$
(g_1,\ldots,g_k) \xlongrightarrow{L_{ij}} (g_1,\ldots,g_j^{-1}g_i,\ldots,g_k).
$$

The graph $\Gamma_k(G)$ is not necessarily connected.
In what follows we always fix an initial generating $k$-tuple $S_0$ (or equivalently an initial epimorphism $e_0$) and consider \marginnote{$\Gamma_k'(G)$} $\Gamma_k'(G)$, the connected component of $\Gamma_k(G)$ containing $S_0$ (for simplicity we omit $S_0$ from notation). The set of vertices of $\Gamma'_k(G)$ is the orbit of $S_0$ under the $\SAut(\F_k)$-action. The \textit{product replacement algorithm} is defined to be a lazy random walk on $\Gamma'_k(G)$ starting at~$S_0$. As is the case for any lazy random walk, it converges to the uniform distribution. 

\begin{remark} We prefer to use lazy random walks in this paper, since the estimates we use in \Cref{t:rate-of-convergence} work only for lazy random walks. Note, however, that $\SAut(F_k)$ has relations of length $2$, namely $L_{ij}L^{-1}_{ij}$, and length $5$, namely $[L^{-1}_{kj},L_{ji}] = L_{ki}$. It follows that $\Gamma'_k(G)$ has loops of length $2$ and $5$ starting from any vertex. Thus $\Gamma'_k(G)$ is aperiodic, and therefore a simple random walk on $\Gamma'_k(G)$ as well converges to the uniform distribution. 
\end{remark}

\begin{example}\label{e:Z^k}
Let $e \colon \F_k \to \BB Z_n^k$, $n \in \BB N$, be the abelianisation modulo $n$, or equivalently $S_0 =(v_1,\ldots,v_k)$ where $v_i \in \BB Z_n^k$ are versors. Then the vertices of $\Gamma'_k(\BB Z_n^k)$ are positive bases of $\BB Z_n^k$, or equivalently elements of $\SL_k(\BB Z_n)$. Members of such bases (called base vectors) are exactly vectors $x = (x_1,\ldots,x_k)$ such that $\OP{gcd}(x_1,\ldots,x_k,n)=1$. The (lazy) random walk on $\Gamma'_k(\BB Z_n^k)$ converges to the uniform distribution on the vertices of $\Gamma'_k(\BB Z_n^k)$. However, if we output, say, the first coordinate of a generating tuple given by the random walk, the distribution of such an element will not converge to the uniform distribution on $\BB Z_n^k$. One can show, that it converges to the uniform distribution on the base vectors of $\BB Z_n^k$.
\end{example}

\begin{example}\label{e:characteristic}
This is a generalization of the above example. Let $e \colon \F_k \to G$ be an epimorphism such that $\ker(e)$ is $\SAut(\F_k)$-invariant. Let $\Gamma_k'(G)$ be the connected component of $e$. Such an epimorphism induces the map $\bar{e} \colon \SAut(\F_k) \to \Aut(G)$ by 
$$\bar{e}(\psi)(e(x)) = e(\psi(x)).$$

Moreover, we have a natural action of $\Aut(G)$ on $\Epi(\F_k,G)$ by composition: if $\bar{\psi} \in \Aut(G)$ and $f \in \Epi(\F_k,G)$ then $\bar{\psi}.f = \bar{\psi} \circ f$. On the vertices of $\Gamma_k'(G)$, the $\SAut(\F_k)$-action factors through the $\Aut(G)$-action via $\bar{e}$ composed with the inverse anti-homomorphism of $\SAut(\F_k)$. Namely, we have:
$$\psi.f = \bar{e}(\psi^{-1}).f \text{ for } \psi \in \SAut(\F_k) \text{ and } f \in \Gamma'_k(G).$$ 

It means, that the vertices of $\Gamma'_k(G)$ are in the orbit of $e$ under the $\bar{e}(\SAut(\F_k))$-action. This action is clearly faithful. Therefore $\Gamma_k'(G)$ with inverted labels on the edges ($L_{ij}$ becomes $L^{-1}_{ij}$ and so forth), is isomorphic to the Cayley graph of $\bar{e}(\SAut(\F_k))$ with the generating tuple $\bar{e}(N_k)$. By a base element of $G$ we mean any element that occurs in a generating $k$-tuple in $\Gamma'_k(G)$ (note that the notion of a base element possibly depends on $e$). One can prove, using the transitivity of $\bar{e}(\SAut(\F_k))$, that the distribution of the first element of a tuple given by the random walk on $\Gamma_k'(G)$ converges to the uniform distribution on the base elements. 
\end{example}

\section{Random walk in a randomly changing Cayley graph}\label{s:random-walk}

Let $\FP$ be the free group of rank $k+1$ where $\BB Z = \la x_0 \ra$, and $\F_k = \la x_1,\ldots, x_k \ra$. Consider the subset $\OP E_k(G)$ of $\Epi(\FP,G)$ consisting of these epimorphisms that are onto already on the subgroup $\F_k$. That is, for every $e \in \OP E_k(G)$, the element $e(x_0)$ is arbitrary and elements $e(x_1), \ldots, e(x_k)$ generate $G$. Thus $\OP E_k(G)$ has the structure of a product set $G \times \Epi(\F_k,G)$. We can also interpret $\OP E_k(G)$ as a set of tuples:

$$
\OP E_k(G) = \{ (g_0 | g_1,\ldots,g_k) \colon g_0 \in G, \la g_1,\ldots,g_k \ra = G \}. 
$$

We write $(g_0 | g_1, \ldots,g_k)$ instead of $(g_0, g_1,\ldots,g_k)$ to underline the different roles of the last $k$ and the first element. The first element $g_0$ of such a tuple is called an accumulator. The last $k$ elements are called generators.

Let \marginnote{$\OP A_k$} $\OP A_k$ be the subgroup of $\SAut(\FP)$ generated by the following two types of automorphisms.

\begin{itemize}[itemindent=4em]
\item[\textbf{$C$-generators:}] $L^{\pm}_{\mathrm{0}i}$ and $R^{\pm}_{\mathrm{0}i}$ for $i=1,\ldots,k$.
\item[\textbf{$N$-generators:}] $L^{\pm}_{ij}$ and $R^{\pm}_{ij}$ for $i,j \in 1,\ldots,k$.
\end{itemize}

We have $4k$ $C$-generators and $4k(k-1)$ $N$-generators. Altogether $\A_k$ is generated by $4k^2$ elements. 

\begin{lemma}
The group $\A_k$ acts on $\E_k(G)$.
\end{lemma}
\begin{proof}
If $(g_0 | g_1,\ldots,g_k) \in \E_k(G)$, then multiplying some $g_i^{\pm}$ where $i=1,\ldots,k$ to $g_0$, leads to an element in $\E_k(G)$. Thus $L^{\pm}_{0i}$ and $R^{\pm}_{0i}$ for $i=1,\ldots,k$ act on $\E_k(G)$. 
One can also perform Nielsen transformations within last $k$ elements. Thus $L^{\pm}_{ij}$ and $R^{\pm}_{ij}$ for $i,j \in 1,\ldots,k$ act on $\E_k(G)$. 
\end{proof}

Let \marginnote{$\Sigma_k(G)$} $\Sigma_k(G)$ be the Schreier graph defined by the action of $\A_k$ on $\E_k(G)$ and the generating set of $\A_k$ consisting of $C$-generators and $N$-generators. 
The edges labeled by $C$-generators are called $C$-edges. The same convention is used for $N$-generators and $N$-edges. 

The graph $\Sigma_k(G)$ comes equipped with a natural subgraph structure reflecting the product decomposition $G \times \Epi(\F_k,G)$ and the distinction between $C$-edges and $N$-edges. Let us describe this decomposition more precisely. Let the projections:
\begin{align*}
\pi_C &\colon \E_k(G) \to G\\
\pi_N &\colon \E_k(G) \to \Epi(\F_k,G)
\end{align*}
be given by 
\begin{align*} 
\pi_C(g_0 | g_1,\ldots,g_k) = & g_0\\
\pi_N(g_0 | g_1,\ldots,g_k) = & (g_1,\ldots,g_k)
\end{align*}

Note that $\pi_N$ induces a map of graphs $\pi_N \colon \Sigma_k(G) \to \Gamma_k(G)$ in the following sense: $N$-edges are sent to the edges of $\Gamma_k(G)$ with their labels preserved, whereas each $C$-edge $e_1 \xlongrightarrow{s} e_2$ is collapsed to the vertex $\pi_N(e_1)=\pi_N(e_2) \in \Gamma_k(G)$.

Let $S = (g_1,\ldots,g_k) \in \Gamma_k(G)$ be a generating $k$-tuple of $G$. 
The directed left-right Cayley graph $Cay(S,G,S)$ is defined as follows: 

\begin{enumerate}
\item
The vertices of $Cay(S,G,S)$ are the elements of $G$.
\item There is a labeled edge $v_1 \xlongrightarrow{L^{\pm}_i} v_2$ if $v_2 = g_i^{\pm}v_1$ and  $v_1 \xlongrightarrow{R^{\pm}_i} v_2$ if $v_2 = v_1g_i^{\pm}$ for $1 \leq i \leq k$.
\end{enumerate}

Note that here $S$ is not a generating set, but a generating $k$-tuple. In particular, generators $g_i$ may be repeated or can be equal the identity of $G$.
Thus $Cay(S,G,S)$ can have multiple edges between vertices and loops. This is why we label edges not by generators but by their indices in $S$. Moreover, the labeling we use corresponds to $C$-labels in $\Sigma_k(G)$.

\begin{lemma}
Let $S \in \Gamma_k(G)$. Consider the subset of vertices 
$$\pi_N^{-1}(S) = \{(g|S) \colon g \in G \} \subset \Sigma_k(G).$$
The subgraph of $\Sigma_k(G)$ spanned on $\pi_N^{-1}(S)$ by $C$-edges is isomorphic to $Cay(S,G,S)$. The labeling is preserved in the sense that edges labeled by $C$-generators $R^{\pm}_{0i}$ or $L^{\pm}_{0i}$ correspond to edges labeled by $R^{\mp}_i$ or $L^{\mp}_i$ in $Cay(S,G,S)$.
\end{lemma}

\begin{lemma}
Let $g \in G$. Consider the subset of vertices 
$$\pi_C^{-1}(g) = \{(g|S) \colon S \in \Gamma_k(G) \} \subset \Sigma_k(G).$$
The subgraph of $\Sigma_k(G)$ spanned on $\pi_C^{-1}(g)$ by $N$-edges is isomorphic to $\Gamma_k(G)$ and the labels are preserved. 
\end{lemma}

Thus $\Sigma_k(G)$ restricted to $C$-edges is a collection of $Cay(S,G,S)$, indexed by vertices of $\Gamma_k(G)$.
If we restrict $\Sigma_k(G)$  to $N$-edges, then for every $g \in G$ we see the same graph $\Gamma_k(G)$ spanned by $N$-edges on $\pi_C^{-1}(g)$.

Recall that $\Gamma'_k(G)$ is the connected component of $\Gamma_k(G)$ containing the given generating $k$-tuple $S_0$. Let \marginnote{$\Sigma'_k(G)$} $\Sigma'_k(G)$ be the connected component of $(1|S_0)$. We have $\pi_N(\Sigma'_k(G)) = \Gamma'_k(G)$.

The random walk on $\Sigma'_k(G)$ starting at $(1|S_0)$ can be described as follows. We imagine that at all times we are in some element of $G$ and on $G$ we see the structure of a left-right Cayley graph. We start at the trivial element $e \in G$ and the graph we see is $Cay(S_0,G,S_0)$. 
When a $C$-edge is chosen, the walk goes along an edge of $Cay(S_0,G,S_0)$.
Selecting an $N$-edge changes the structure of the Cayley graph to some other $Cay(S,G,S)$. Then for some time we walk in $Cay(S,G,S)$ until we change the Cayley graph again. 
Thus walking in $\Sigma'_k(G)$ is interpreted as a random walk in a randomly changing Cayley graph. 

\begin{remark}\label{r:only-left-transvections}
Let \marginnote{$\OP{LA}_k$} $\OP{LA}_k$ be defined like $\A_k$, but for $C$-generators we take only left Nielsen automorphisms $L_{0i}^\pm$. In the description of $\Sigma_k(G)$   for $\OP{LA}_k$ we would use standard Cayley graphs in place of left-right Cayley graphs. All theorems we show in \Cref{s:main-result} are valid for $\OP{LA}_k$ with the same constants. We prefer to use $\A_k$ due to the symmetry of its generators.  
\end{remark}

\section{Product replacement accumulator algorithm}\label{s:praa}

Let $\Gamma$ be a graph. The lazy random walk on $\Gamma$ is a random walk that with probability $\frac{1}{2}$ stays put and with probability $\frac{1}{2}$ moves to a uniform neighbor. Let $G$ be a finite group and let $S_0$ be a given generating $k$-tuple of $G$. We consider the following algorithm, which is a slight modification of the ``rattle'' algorithm from \cite{MR1929718}.

\textbf{Product replacement accumulator algorithm for $(G,S_0)$:}

\begin{enumerate}
\item Start from $(1|S_0) \in \Sigma'_k(G)$.
\item Perform a lazy random walk with $t$ steps on $\Sigma_k'(G)$.
\item Return the accumulator (the first element of the resulting tuple).
\end{enumerate}

Let $\mu_t$ be the distribution of the lazy random walk on $\Sigma_k'(G)$ after $t$ steps, and let $\nu_t = \pi_C^*(\mu_t)$ be the distribution on $G$ of the accumulator. Note that while the lazy random walk from (2) defines a Markov chain, the sequence of random variables given by the accumulator is no longer Markov.  

Before we proceed let us recall some classical definitions. Let $X$ be a finite set and let $U_X$ be the uniform probability distribution on $X$, i.e. $U_X(x) = \frac{1}{|X|}$ for every $x \in X$. 

For two probability distributions $\mu$ and $\nu$ on $X$ we define the total variation distance:

$$
\|\mu-\nu\|_{\mathrm{tv}} = \frac{1}{2}\sum_{x\in X} |\mu(x)-\nu(x)|.
$$

Let $A \colon X \to Y$ be a map between finite sets.
By $A^*(\mu)$ we denote the push-forward of $\mu$, i.e., $A^*(\mu)(y) = \mu(A^{-1}(y))$.

\begin{lemma}\label{l:A-contracting}
$A^*$ is a contraction, that is
$$\|A^*(\mu)-A^*(\nu)\|_{\mathrm{tv}} \leq \|\mu-\nu\|_{\mathrm{tv}},$$

where $\mu$ and $\nu$ are distributions on $X$.
\end{lemma}

\begin{lemma}\label{l:acc-conv-to-uniform}
Distributions $\nu_t$ converge to the uniform distribution on $G$ and moreover
$$
\|\nu_t - U_G\|_{\mathrm{tv}} \leq \|\mu_t - U_{\Sigma'_k(G)}\|_{\mathrm{tv}}.
$$
\end{lemma}

\begin{proof}
We have that  $|\pi_C^{-1}(g)| = |\Gamma'_k(G)|$. Hence it is independent of $g \in G$. Thus

$$
\pi_C^*(U_{\Sigma'_k(G)})(g) = 
U_{\Sigma'_k(G)}(\pi_C^{-1}(g)) =
\frac{|\Gamma'_k(G)|}{|\Sigma'_k(G)|} = 
\frac{1}{|G|}
$$

Hence $\pi_C^*(U_{\Sigma'_k(G)}) = U_G$. Now we apply \Cref{l:A-contracting} and the fact that $\mu_t$ converges to the uniform distribution on $\Sigma_k'(G)$.
\end{proof}

\section{Proof of the main result}\label{s:main-result}

Let $R$ be a ring and $* \colon R \to R$ be a map. We say that $\xi \in R$ is a sum of squares, and write  \marginnote{$\xi \geq 0$} $\xi \geq 0$, if $\xi = \Sigma_i \xi_i^*\xi_i$ for some $\xi_i \in R$. In this paper we have $R = \BB R[G]$, the group ring of a group $G$, where $*$ is the linear extension of $g^*=g^{-1}$, or $R = \OP{L}(\BB C^V)$, the ring of linear endomorphisms of $\BB C^V$, with $*$ being the Hermitian transposition. 

To show our main result we use the strategy developed in \cite{MR4224715}. First, we carefully select an element $\xi \in \mathbb{R}[\A_5]$ and use computer calculations to show that $\xi\geq 0$. Next, using human calculations and the fact that $\xi \geq 0$, we show for $k> 5$ that $\Delta^2_k - 0.35(k-5)\Delta_k \geq 0$ and $\Delta_5^2-1.41\Delta_5\geq 0$, where $\Delta_k$ is the group Laplacian of $\A_k$. From that result  \Cref{t:rate-of-convergence} follows easily. 

\subsection{Decomposition of the squared Laplacian.}\label{ss:decompositions}

Let $k \in \mathbb{N}$. For every pair of indices $i \neq j$ we define the partial Laplacian $\Delta_{ij} \in \mathbb{R}[\A_k]$ to be:
$$\Delta_{ij} = (1-L_{ij})(1-L_{ij})^* + (1-R_{ij})(1-R_{ij})^* .$$ 

We also define $\Delta_{ii} = 0$. The Laplacians related to $C$-generators and $N$-generators are: 

$$
\Delta_k^C = \sum_{1 \leq s \leq k} \Delta_{\mathrm{0}s}
\hspace{1cm}
\Delta_k^N = \sum_{1 \leq i,j \leq k} \Delta_{ij}.
$$

The full non-normalized group Laplacian of $\A_k$ is \marginnote{$\Delta_k$}:

$$
\Delta_k = \Delta_{k}^C + \Delta_k^N.
$$

Now we focus on the square of the Laplacian. We have 

$$
\Delta_k^2 = (\Delta_k^C)^2+ \{\Delta_k^C,\Delta_k^N\}+ (\Delta_k^N)^2,
$$

where 
$$
\{\Delta_k^C,\Delta_k^N\}=\Delta_k^C\Delta_k^N+\Delta_k^N\Delta_k^C.
$$

We shall further decompose each of the above terms. In the following sums the indices always range from $1$ to $k$.

\underline{\textit{Decomposition of $(\Delta_k^C)^2$:}}

We have
$$
(\Delta_k^C)^2 = \Sq_k^C + \Adj_k^C,
$$

where 

\[
\begin{array}{lll}
\Sq_k^C &= &\displaystyle \sum_{s}\Delta_{\mathrm{0}s}^2 \\[1.5em]
\Adj_k^C &=&\displaystyle  \sum_{s\neq t}\Delta_{\mathrm{0}s}\Delta_{\mathrm{0}t}.
\end{array}
\]

\underline{\textit{Decomposition of $\{\Delta_k^C,\Delta_k^N\}$:}}

We have 

$$
\{\Delta_k^C,\Delta_k^N\} = \Adj_k^{CN} + \Opp_k^{CN},
$$

where 

\[
\begin{array}{lll}
\Adj_k^{CN} &=& \;\;\;\displaystyle \sum_{i,j}\;\;\;\;\{\Delta_{\mathrm{0}i}+\Delta_{\mathrm{0}j},\Delta_{ij}\}\\[1.5em]
\Opp_k^{CN} &=& \displaystyle \sum_{|\{s,i,j\}|=3}\{\Delta_{\mathrm{0}s},\Delta_{ij}\}.
\end{array}
\]

\underline{\textit{Decomposition of $(\Delta_k^N)^2$:}}

We have

$$
(\Delta_k^N)^2 = \Sq_k^N+\Adj_k^N+\Opp_k^N,
$$

where 

\[
\begin{array}{lll}
\Sq_k^N &=& \; \displaystyle \frac{1}{2}\sum_{i,j}\;\;(\Delta_{ij}+\Delta_{ji})^2\\[1.5em]
\Adj_k^N &=& \;\displaystyle  \sum_{|\{i,j,l\}|=3}\Delta_{ij}(\Delta_{il}+\Delta_{li}+\Delta_{jl}+\Delta_{lj})\\[1.5em]
\Opp_k^N &=& \displaystyle \sum_{|\{i,j,l,m\}| = 4\!\!\!\!}\Delta_{ij}\Delta_{lm}.
\end{array}
\]

\subsection{Induction}\label{ss:induction}

Let $k \leq K \in \BB N$ and let $S_K$ act on $\{x_1,\ldots,x_K\}$ by permuting indices. This action induces an action on $\F_{K+1}$ by automorphisms (the generator $x_0$ is fixed), which in turn induces the action by conjugations $\psi \to \sigma\psi\sigma^{-1}$ on $\A_K$. It follows that $\sigma.L_{ij} = L_{\sigma(i)\sigma(j)}$, and $\sigma.R_{ij} = R_{\sigma(i)\sigma(j)}$. 

Let $\xi \in \BB R[\A_k]$. We define the symmetrization $S \colon \BB R[\A_k] \to \BB R[\A_K]$ to be:

$$
S(\xi) = \sum_{\sigma \in S_K} \sigma.\xi.
$$

Now we apply $S$ to the elements that appear in the decompositions in  \Cref{ss:decompositions}. For $\Delta_k$ we have:

\[
\begin{array}{l@{\;}c@{\;}l@{\;\;}l}
S(\Delta_k^C)   & = & (K-1)!\,k\,\Delta_K^C\\
S(\Delta_k^N)   & = & (K-2)!\,(k-1)k\,\Delta_K^N.
\end{array}
\]

For $(\Delta_k^C)^2$ we have:

\[
\begin{array}{l@{\;}c@{\;}l@{\;\;}l}
S(\Adj_k^C)     & = & (K-2)!\,(k-1)k\,\Adj_K^C.
\end{array}
\]

For $\{\Delta_k^C,\Delta_k^N\}$ we have:

\[
\begin{array}{l@{\;}c@{\;}l@{\;\;}l}
S(\Adj_k^{CN})  & = & (K-2)!\,(k-1)k\,\Adj_K^{CN}\\
S(\Opp_k^{CN})  & = & (K-3)!\,(k-2)(k-1)k\,\Opp_K^{CN}.
\end{array}
\]

And finally for $(\Delta_k^N)^2$ we have:

\[
\begin{array}{l@{\;}c@{\;}l@{\;\;}l}
S(\Sq_k^N)      & = & (K-2)!\,(k-1)k\,\Sq_K^N\\
S(\Adj_k^N)     & = & (K-3)!\,(k-2)(k-1)k\,\Adj_K^N\\
S(\Opp_k^N)     & = & (K-4)!\,(k-3)(k-2)(k-1)k\,\Opp_K^N.
\end{array}
\]

\begin{lemma}\label{l:induction-element}
The following inequality holds in $\BB R[\A_5]$ with $\lambda_5 = 1.41$:
$$\Adj_5^C + \Adj_5^{CN} + \Opp_5^{CN} + (\Delta_5^N)^2 \geq \lambda_5\Delta_5.$$
\end{lemma}

\begin{proof}
Note that $\Delta_{ij}$ is contained in the augmentation ideal and $\Delta_{ij} = \Delta_{ij}^*$. Thus the same is true for $\xi = \Adj_5^C + \Adj_5^{CN} + \Opp_5^{CN} + (\Delta_5^N)^2$. Now we can find a numerical approximation and certify the inequality as described in \Cref{s:computer-calc}.
\end{proof}

\begin{theorem}\label{t:sos}
Assume that $K > 5$ and let $\Delta_K$ be the non-normalized Laplacian of $\A_K$. We have $\Delta_K^2 - 0.35(K-5)\Delta_K \geq 0$. For $K=5$ we have $\Delta_5^2 - 1.41\Delta_5 \geq 0.$
\end{theorem}

\begin{proof}

Let $k=5$ and $K \geq 5$. Using \Cref{l:induction-element}, applying the operator $S$ and dividing by $(K-2)!(k-1)k$ we obtain:

\begin{equation}
\begin{array}{l@{\hskip 5pt}c@{\hskip 5pt}r@{\hskip 5pt}c@{\hskip 5pt}l@{\hskip 5pt}r@{\hskip 5pt}r@{\hskip 5pt}l@{\hskip 5pt}}
&&&&\Adj_K^{C}\\[0.4em]
+&&&&\Adj_K^{CN}&+&\frac{k-2}{K-2}&\Opp_K^{CN}\\[0.4em]
+&\Sq_K^{N}&+&\frac{k-2}{K-2} &\Adj_K^{N}&+&\frac{(k-3)(k-2)}{(K-3)(K-2)}& \Opp_K^{N} \\[0.6em]
&&&&&\geq &\frac{\lambda_k (K-1)}{k-1}&\Delta_K^C+\lambda_k\Delta_K^N. 
\end{array}
\label{eq:main}
\end{equation}

Since $K-2 \geq k-2$, the coefficient of $\Opp_K^{CN}$ is not greater than one. Note that $\Opp_K^{CN} \geq 0$ and $\Opp_K^N \geq 0$ \cite[Lemma 3.6]{MR4224715}.\\

By \cite[Remark 5.8]{MR4224715} we have  $\Adj_k^N + 3\Opp_k^N \geq 1.37\Delta_k^N$. Using the operator $S$ and dividing by $(K-3)!(k-2)(k-1)k$ we get 

\begin{equation}
\Adj_K^N + \frac{3(k-3)}{K-3}\Opp_K^N  \geq \frac{1.37(K-2)}{k-2}\Delta_K^N.
\label{eq:kkn}
\end{equation}

Let $u = 1 - \frac{k-2}{K-2} \geq 0$. We multiply inequality \ref{eq:kkn} by $u$ and add to inequality \ref{eq:main}. Observe that then the coefficient of $\Adj_K^N$ beomes $1$ and the coefficient of $\Opp_K^N$ is: 
$$\gamma = \frac{3(k-3)u}{K-3} +\frac{(k-3)(k-2)}{(K-3)(K-2)},$$ 

which is not greater than $1$ by a straightforward computation. Thus we get:

\[
\begin{array}{r@{\hskip 5pt}c@{\hskip 5pt}l@{\hskip 5pt}l@{\hskip 5pt}l@{\hskip 5pt}r@{\hskip 5pt}l}
&&&\Adj_K^{C}\\[0.4em]
+&&&\Adj_K^{CN}  & + & \frac{k-2}{K-2}&\Opp_K^{CN} \\[0.4em]
+& \Sq_K^{N}   & + & \Adj_K^{N}&+&\gamma &\Opp_K^{N} \\[0.4em]
&& && \geq&  \frac{\lambda_k (K-1)}{k-1}&\Delta_K^C + (\frac{1.37 (K-2)u}{k-2}+\lambda_k)\Delta_K^N.
\end{array}
\]

Next, we increase the coefficient of $\Opp^{CN}_K$ and $\Opp^N_K$ to $1$ and add $\Sq_K^C \geq 0$. We get:

\begin{equation}
\begin{array}{r@{\hskip 5pt}c@{\hskip 5pt}l@{\hskip 5pt}l@{\hskip 5pt}l@{\hskip 5pt}r@{\hskip 5pt}l}
(\Delta_K^C+\Delta_K^N)^2 & = &\Sq_K^C&+&\Adj_K^{C}\\[0.4em]
&+&&&\Adj_K^{CN}  & + &\Opp_K^{CN} \\[0.4em]
&+& \Sq_K^{N}   & + & \Adj_K^{N}&+&\Opp_K^{N} \\[0.6em]
&&& && \geq&  \frac{\lambda_k (K-1)}{k-1}\Delta_K^C + (\frac{1.37 (K-2)u}{k-2}+\lambda_k)\Delta_K^N.
\end{array}
\label{eq:final}
\end{equation}

Since $(K-2)u = K-5$, $k = 5$ and $\lambda_k = 1.41$, we have 

$$
\frac{1.37(K-5)}{k-2}+\lambda_k > \frac{\lambda_k(K-5)}{k-1} = 0.35(K-5).
$$ 

Therefore $\Delta_K^2 \geq 0.35(K-5)\Delta_K$.

If $K=5$, we have $u=0$ and directly from inequality \ref{eq:final} it follows that $\Delta^2_K \geq 1.41\Delta_K$. 

\end{proof}

Note that it follows from \Cref{t:sos} that $\A_k$ has property (T) for $k\geq 5$ \cite{MR3427595}.

\begin{remark}
To show \Cref{t:sos}, it would be enough to have $\Adj_5^C + \Adj_5^{CN} + \Opp_5^{CN} \geq \lambda_5\Delta_5$ for some $\lambda_5>0$. However, computer calculations could not certify this inequality. This is why we added the additional term $(\Delta_5^N)^2$.
\end{remark}

\subsection{Spectral gap estimates}\label{ss:convergence}

Let $G$ be a finite group and $S_0$ be a generating $k$-tuple of $G$. Let $\Sigma'_k(G)$ be the connected component of $(1|S_0)$ in $\Sigma_k(G)$. Recall that $\Sigma'_k(G)$ is a $4k^2$-regular graph. By \marginnote{$\Delta_{\Sigma'_k(G)}$} $\Delta_{\Sigma'_k(G)}$ we denote the non-normalized Laplacian of $\Sigma'_k(G)$, i.e., $\Delta_{\Sigma'_k(G)} = 4k^2I - A$ where $A$ is the adjacency matrix of $\Sigma'_k(G)$. Each entry of $A$ is the number of directed edges connecting a given vertex to a given vertex. The spectral gap of $\Sigma'_k(G)$ is the spectral gap of the normalized Laplacian $\frac{1}{4k^2}\Delta_{\Sigma'_k(G)}$.

\begin{corollary}\label{c:spectral-gap}
Let $k>5$. The spectral gap of $\Sigma'_k(G)$ is at least $\frac{0.35(k-5)}{4k^2}$.
\end{corollary}

\begin{proof}
Let $V$ be the set of vertices of $\Sigma'_k(G)$. Since $\A_k$ acts on $V$, we have a unitary representation $\rho_G \colon \A_k \to \OP{U}(\BB C^V)$. By $\rho_G$ we denote also its linear extension to the ring $*$-homomorphism $\rho_G \colon \BB R [\A_k] \to \OP{L}(\BB C^V)$ where by $\OP{L}(\BB C^V)$ we denote the space of linear endomorphisms of $\BB C^V$. We have that 
$$\rho_G(\Delta_k) = \Delta_{\Sigma'_k(G)}.$$
Note that if $\xi \in \OP{L}(\BB C^V)$ and $\xi^2-\lambda \xi \geq 0$, then the spectral gap of $\xi$ is at least $\lambda$. By  \Cref{t:sos}, we have $\Delta_{\Sigma'_k(G)}^2 - 0.35(k-5)\Delta_{\Sigma'_k(G)} \geq 0$, hence the spectral gap of $\Sigma'_k(G)$ is at least $\frac{0.35(k-5)}{4k^2}$.
\end{proof}

\begin{remark}\label{r:sharp}
The growth of $\lambda_k$ in the inequality $\Delta_k^2 - \lambda_k\Delta_k \geq 0$ in \Cref{t:sos} cannot be faster than linear. Indeed, $\A_k$ maps onto $\SAut(\F_k)$, see  \Cref{s:semidirect}, and $\SAut(\F_k)$ maps further to $\OP{SL}_k(\BB Z)$ which acts transitively on $V = \BB Z_3^k \setminus \{0\}$. This action induces a representation $\rho_k \colon \A_k \to \OP{U}(\BB C^V)$. 
Note that for $v = \Sigma_{i=1}^k (\delta_{e_i} - \delta_{-e_i}) \in \BB C^V$, we have: 
$$\sup_{s} \|\rho_k(s)v-v
\|= \sqrt{\frac{2}{k}}\|v\|,$$

where $s$ runs over the generators of $\A_k$. Let $\Delta_{\rho_k} = \rho_k(\Delta_k)$ and let $\lambda_{\rho_k}$ be the spectral gap of $\Delta_{\rho_k}$. By \Cref{c:spectral-gap} we have $\lambda_k \leq \lambda_{\rho_k}$. Since $v$ is orthogonal to constant functions in $\BB C^V$, and constant functions are the only eigenvectors of $\Delta_{\rho_k}$ with $0$ eigenvalue, we have: 

$$
\lambda_{\rho_k} \leq \frac{\|\la \Delta_{\rho_k} v,v \ra \|}{\|v\|^2} = \frac{1}{\|v\|^2}\sum_{s}\|\rho_k(s)v-v\|^2 \leq 4k^2\Big(\sqrt{\frac{2}{k}}\Big)^2 = 8k.
$$ 

Thus \Cref{t:sos} is asymptotically sharp. 
\end{remark}

\begin{theorem}\label{t:rate-of-convergence}
Let $G$ be a finite group and $S_0$ a generating $k$-tuple of $G$, $k>5$. Let $\nu_t$ be the distribution of the product replacement accumulator for $(G,S_0)$ after $t$ steps. We have

$$
\|\nu_t - U_G\|_{\mathrm{tv}} < \epsilon \text{ for } t \geq \frac{8k^2}{0.35(k-5)}((k+1)\log|G| + \log(\epsilon^{-1})).
$$
\end{theorem}

\begin{proof}
Let $\mu_t$ be the distribution of the lazy walk on $\Sigma'_k(G)$ after $t$ steps. \Cref{c:spectral-gap} and the standard bound relating the rate of convergence to the uniform distribution and the spectral gap of $\Sigma'_k(G)$ \cite[Chapter 12]{MR2466937} give:

$$
\|\mu_t - U_G\|_{\mathrm{tv}} < \epsilon \text{ for } t \geq \frac{8k^2}{0.35(k-5)}\log\Big(\frac{|\Sigma'_k(G)|}{\epsilon}\Big).
$$

The theorem follows from \Cref{l:acc-conv-to-uniform} and the inequality $|\Sigma'_k(G)| < |G|^{k+1}$.

\end{proof}

\begin{remark}
A similar argument for $k=5$ gives

$$
\|\nu_t - U_G\|_{\mathrm{tv}} < \epsilon \text{ for } t \geq \frac{8k^2}{1.41}((k+1)\log|G| + \log(\epsilon^{-1})).
$$

\end{remark}

\begin{corollary}
Fix $k \geq 5$. The family of $4k^2$-regular graphs $\Sigma'_k(G)$, indexed by $(G,S_0)$ where $G$ is arbitrary finite group and $S_0$ is a generating $k$-tuple of $G$, is an expander. 
\end{corollary}

\begin{remark}
\Cref{l:induction-element} is true as well for $\OP{LA}_5$ with the same constant, i.e. $\lambda_5 = 1.41$ and appropriately defined operators. \Cref{t:sos} and \Cref{t:rate-of-convergence} remain valid with the same constants as well. The proofs apply verbatim. 
\end{remark}

\section{Alternative proof of property (T)}\label{s:semidirect}

In this complementary paragraph we show that $\A_k$ is a semidirect product and discuss a different argument showing that $\A_k$ has the property (T). However, this argument gives much worse estimates than \Cref{t:rate-of-convergence}.

\begin{lemma}
The group $\A_k$ is isomorphic to $\F_k^2 \rtimes \SAut(\F_k)$, where the action of $\SAut(\F_k)$ on $\F_k^2~=~\F_k \times \F_k$ is the direct product of the canonical action. 
\end{lemma}

\begin{proof}
Let $\F_k < \FP$, where $\F_k = \la x_1,\ldots,x_k \ra$.
Note that $N$-generators generate $\SAut(\F_k)$, thus $\SAut(\F_k) < \A_k$.
Moreover, elements of $\A_k$ fix $\F_k$, hence by restricting to this subgroup, we have the retraction:
$$
p \colon \A_k \twoheadrightarrow \SAut(\F_k).
$$

Thus $\A_k$ is isomorphic to $\ker(p) \rtimes \SAut(\F_k)$, where $\SAut(\F_k)$ acts on $\ker(p)$ by conjugation. 
Let $\psi \in \ker(p)$, then $\psi$ fixes all $x_1,\ldots,x_k$ and since it is generated by $C$-generators and $N$-generators, to $x_0$ we can multiply only elements from $\F_k$ on the right and left. Thus there exist unique elements $l_{\psi}$ and $w_{\psi}$ in $\F_k$ such that:  
$$\psi(x_0,x_1,\ldots,x_k) = (l_{\psi}x_0r_{\psi},x_1,\ldots,x_k).$$

The map $j(\psi) = (l_\psi^{-1},r_{\psi})$ is an isomorphism between $\ker(p)$ and $\F_k^2$. Indeed, it is injective and since we have all generators $L^{\pm}_{0i}$ and $R^{\pm}_{0i}$, we can obtain every element in $\F_k^2$. Moreover, since
$$
\psi_1\psi_2(x_0) = \psi_1(l_{\psi_2}x_0r_{\psi_2}) = l_{\psi_2}l_{\psi_1}x_0r_{\psi_1}r_{\psi_2},
$$

we have that $j(\psi_1\psi_2) = (l_{\psi_1}^{-1}l_{\psi_2}^{-1},r_{\psi_1}r_{\psi_2}) = j(\psi_1)j(\psi_2)$, thus $j$ is a homomorphism. It follows that $\A_k$ is isomorphic to $\F_k^2 \rtimes \SAut(\F_k)$.

Now we compute how the conjugation action of $\SAut(\F_k)$ on $\ker(p)$ looks like after identifying $\ker(p)$ with $\F_k^2$ via the isomorphism $j$. Let $\phi \in \SAut(\F_k) < \A_k$ and $\psi \in \ker(p)$ be such that $j(\psi) = (l^{-1},r)$. Note that $\phi(x_0) = x_0$. We have 

\begin{align*}
\phi \psi \phi^{-1} (x_0) & =\phi \psi (x_0)\\
& = \phi (lx_0r)\\
&=\phi(l)x_0\phi(r).
\end{align*}

Thus $j(\phi \psi \phi^{-1}) = (\phi(l^{-1}),\phi(r))$ and the action of $\SAut(\F_k)$ on $\F_k^2$ is the direct product of the canonical action.

\end{proof}

In an analogous way one can show that $\OP{LA}_k$ is isomorphic to $\F_k \rtimes \SAut(\F_k)$, see \Cref{r:only-left-transvections} for the definition of $\OP{LA}_k$.

If one is interested only in the property (T) for $\A_k$ or $\OP{LA}_k$, one could use directly \cite[Proposition 10]{MR4023374}. However, this approach gives much worse estimates of the spectral gaps than what follows from  \Cref{t:rate-of-convergence}. Let us compute these estimates. 

Assume that $G$ is a group and $S$ is a fixed generating set of $G$. The Kazhdan's constant of $(G,S)$ is the largest number $\kappa(G)$ such that 
$$
\max_{s\in S} \|\rho(s)v-v\| \geq \kappa(G)\|v\|,
$$
for every irreducible unitary representation $\rho$ of $G$ and every $v \in \mathcal{H}_\rho$. 

Suppose $k \geq 4$. It follows from \cite[Remark 5.11]{MR4224715} for $k\geq 5$, and from \cite{nitsche2022computerproofspropertyt} for $k=4$, that $\SAut(\F_k)$ has property (T) and the Kazhdan's constant $\kappa(\SAut(\F_k))$ is of order $\frac{1}{\sqrt{k}}$. By \cite[Proposition 10]{MR4023374}, we have that $\kappa(\OP{LA}_k)$ is at least of order $ k^{-1}$. Since $\A_k = \F_k \rtimes \OP{LA}_k$, using again \cite[Proposition 10]{MR4023374}, we have $\kappa(\A_k)$ is at least of order $k^{-1.5}$. Now we have to estimate the spectral gap of $\Sigma'_k(G)$ for $\A_k$ and the spectral gap of an analogous graph for $\OP{LA}_k$. To our knowledge, the best estimate bounding the spectral gap from below by the Kazhdan's constant is given in \cite{PakZuk}. To use it, we need a subgroup of the automorphism group of $\A_k$ that acts on the generators of~$\A_k$. The maximal such a subgroup is $\Phi = \BB Z_2^{k+1} \rtimes S_k$, see \Cref{ss:num-approx} for more details. 

The smallest orbit of the $\Phi$-action on the generators of $\A_k$ has $4k$ elements, hence the constant $\alpha$ from \cite{PakZuk} equals $\alpha(\A_k) = \frac{4k}{4k^2} = \frac{1}{k}$. A similar analysis for $\OP{LA}_k$, where $\Phi = \BB Z_2^k \rtimes S_k$, gives $\alpha(\OP{LA}_k) = \frac{1}{2k-1}$. Thus the lower bound for the spectral gap of $\Sigma'_k(G)$ that we get is at least $\alpha(\A_k)\kappa(\A_k)^2$, which gives the bound of order $k^{-4}$ (note that the Kazhdan's constant in \cite{PakZuk} is the square of the Kazhdan's constant as we defined it). For $\OP{LA}_k$ we have the bound of order $k^{-3}$. These estimates are much worse than what follows from \Cref{t:sos} and what is used in \Cref{t:rate-of-convergence}. The main point of our strategy in this paper is that we do not pass through the Kazhdan's constant, but estimate the spectral gap directly from computer calculations. 
 
\section{Turning computer calculations into rigorous proofs}\label{s:computer-calc}
In this section we describe how we obtained the estimate $\lambda_5=1.41$ in \Cref{l:induction-element} and the analogous result for $\OP{LA}_5$. The description is mostly based on the description from \cite{MR3794918}.

Let $G$ be a group generated by a finite symmetric set $S$ which does not contain elements of order two and let $\Delta$ be the Laplacian of $G$ with respect to $S$. Ozawa showed that $\Delta$ is an \emph{order unit} in the augmentation ideal $I[G]$, meaning that for each $*$-invariant $\xi\in I[G]$ there exists $R_{\xi}\geq 0$ such that $\xi+R_{\xi}\Delta\geq 0$. More precisely, the definition of $\Delta$ being an order unit for $I[G]$ actually requires that $\xi+R_{\xi}\Delta=\sum_i\xi_i^*\xi_i$ for some finite number of $\xi_i$ from $I[G]$. However, since $\Sigma^2I[G]=\Sigma^2\mathbb{R}[G]\cap I[G]$, we may just require $\xi_i\in\mathbb{R}[G]$. 

Let $\xi=\xi^*\in I[G]$. Suppose we want to show that there exists a positive $\lambda$ such that $\xi-\lambda\Delta\geq 0$. We start with estimating numerically $\lambda_0>0$ and a sum of squares decomposition of $\xi-\lambda_0\Delta$. Obviously, we cannot search for the solution in the entire infinite group due to finiteness constraints of computers. The remedy for that is to restrict our attention to some finite subset $E\subset G$ and look for the decompositions
\begin{equation}\label{eq:Rxi-bound}
    \xi-\lambda_0\Delta\approx \sum_i\xi_i^*\xi_i,
\end{equation}
where the support of each $\xi_i$ is contained in $E=\{1,g_1,\ldots,g_m\}$. One typically takes $E$ to be a ball of some small radius. One can then write equation \ref{eq:Rxi-bound} equivalently as
$$
\xi-\lambda_0\Delta\approx\mathbbm{x}^*Q^TQ\mathbbm{x}
$$
for some $m\times m$ real matrix $Q$ and 
$$
\mathbbm{x}=\begin{bmatrix}
    1-g_1\\
    \vdots\\
    1-g_m
\end{bmatrix},
$$
where $*$ denotes this time the composition of the group ring $*$ operation with matrix transposition. 

Put $r=\xi-\lambda_0\Delta-\mathbbm{x}^*Q^TQ\mathbbm{x}$. We have
$$
\xi-(\lambda_0-R_r)\Delta=\mathbbm{x}^*Q^TQ\mathbbm{x}+(r+R_r\Delta).
$$

The latter expression is already a sum of squares which would conclude the proof. However, we have to take care of the following aspects:
\begin{enumerate}
    \item First, we have to ensure that $R_r<\lambda_0$ and so far we have not provided any upper-bound for $R_r$. Fortunately, we have to our disposal the following result of T. Netzer and A. Thom:
\begin{theorem}\label{theorem:netzer_thom}\cite[Lemma 2.1]{MR3359223}
    For any $\xi=\xi^*\in I[G]$, there exists $R_{\xi}\geq 0$ such that $\xi+R_{\xi}\Delta\geq 0$ and:
    $$
    R_{\xi}\leq m^2||\xi||_1,
    $$
    where $||\eta||_1=\sum_g|\xi_g|$ denotes the $\ell_1$-norm of $\eta=\sum_g\eta_gg$ and $m$ is the maximum of the world-length norms with respect to $S$ of $g$ from the support of $\xi$.
\end{theorem}

\item In order to ensure the rigor of computations after obtaining the numerical approximations, one performs them in the \emph{interval arithmetic}. The idea is that the coefficients of the group ring elements are no longer ordinary real numbers but small intervals guaranteeing that the actual coefficients lie within the corresponding intervals. All the arithmetic operations preserve that rigor.
\end{enumerate}

All the steps above were implemented in \cite{kaluba_PropertyTjl_2024} and applied before in \cite{MR3794918,MR4023374,MR4224715,MR4912919}.

\subsection{Finding the numerical approximation}\label{ss:num-approx}
In order to obtain numerical approximations of $\lambda$ and $Q$, one applies semi-definite positive programming. We skip the description of that since it is thoroughly described in \cite{MR4023374}. As noted there, one can apply the symmetry of the expression $\Delta^2-\lambda\Delta$ with respect to the generating set given by Nielsen transvections. \emph{Wedderburn decomposition} allows one then to drastically reduce the complexity of the problem. This applies to our case as well. In order to apply this approach, we have to be sure that for our group $G\in\{\OP{A}_k,\OP{LA}_k\}$, the Laplacian $\Delta$ and the finite support set $E$ are $\Phi$-invariant, where $\Phi$ denotes a specific \emph{wreath product}. 
\begin{enumerate}  
    
\item In the case $G=\OP{A}_k$, $\Phi$ is the subgroup of $\mathbb{Z}_2\wr S_{k+1}\cong \mathbb{Z}_2^{k+1}\rtimes S_{k+1}$ generated by the elements of the form $((a_0,\ldots,a_k),\sigma)$ for $\sigma\in S_{k+1}$ fixing $0$ and $S_{k+1}$ denoting the permutation group of the set $\{0,\ldots,k\}$. The action of $\Phi$ on $G$ is defined as follows
\[
((a_0,\ldots,a_{k}),\sigma)T_{i,j}=
\begin{cases}
  (T_{\sigma(i),\sigma(j)})^{\pm a_{\sigma(j)}}, & x_{\sigma(i)}=1 \\
  (T'_{\sigma(i),\sigma(j)})^{\mp a_{\sigma(j)}}, & x_{\sigma(i)}=-1,
\end{cases}
\]
where $T'$ denotes $R$ for $T=L$ and $L$ for $T=R$. The action of an element $((a_0,\ldots,a_{k}),\sigma)$ is precisely the conjugation by the automorphism of the free group $\F_{k+1}=\langle x_0,\ldots,x_k\rangle$ sending $x_i$ to $x_{\sigma(i)}^{a_{\sigma(i)}}$.

 \item In the case of $G=\OP{LA}_k$, $\Phi=\mathbb{Z}_2\wr S_k\cong \mathbb{Z}_2^k\rtimes S_k$, where $S_k$ denotes the permutation group of the set $\{1,\ldots,k\}$. We define the action of $\Phi$ on $T_{i,j}$ by the same formulae as in the case above for $i>0$. For $i=0$ we do the same but we neglect flipping the symbol $T$. This is important since, inheriting directly the conjugation action of the direct product would flip $T$ and we are not allowed to do that due to the lack of the generators $R_{0j}$.
\end{enumerate}
In both cases above, we take $E$ to be some ball of finite radius. It is apparent that $E$ is $\Phi$-invariant.

The above numerical approximations can be obtained in moderate time (in our case not longer than $2$ hours) on a typical desktop computer with $32$ GB of RAM memory . 

\subsection{Replication of the results}
We refer the reader to the GitHub repository with our code \cite{mizerka_pra_complexity_2025} or the Zenodo resource \cite{marcinkowski_mizerka_2025_17981235}. We provide as well the precomputed numerical approximations of $\lambda$ and $Q$. One can use them directly in the certification process -- one can skip then the whole numerical approximation part. The proof remains rigorous, since it is not important for this purpose how actually the approximations of $\lambda$ and $Q$ were obtained.

The Julia \cite{Julia-2017} code we implemented \cite{mizerka_pra_complexity_2025} makes important use (in some cases, indirectly, through other packages) of the following packages: \href{https://github.com/kalmarek/Groups.jl}{Groups.jl} \cite{MR4023374,MR4224715}, \href{https://github.com/JuliaIntervals/IntervalArithmetic.jl}{IntervalArithmetic.jl} \cite{Sanders2023IntervalArithmetic}, \href{https://jump.dev/}{JuMP.jl} \cite{Lubin2023JuMP}, \href{https://github.com/cvxgrp/scs}{SCS} \cite{ODonoghue2016Conic}, \href{https://github.com/JuliaAlgebra/StarAlgebras.jl}{StarAlgebras.jl} \cite{MR4023374,MR4224715}, \href{https://github.com/kalmarek/SymbolicWedderburn.jl}{SymbolicWedderburn.jl} \cite{MR4023374,MR4224715}, and \href{https://github.com/kalmarek/PropertyT.jl}{PropertyT.jl} \cite{kaluba_PropertyTjl_2024}.

\bibliographystyle{abbrvurl}
\bibliography{bib}

@misc{marcinkowski_mizerka_2025_17981235,
  author       = {Marcinkowski, Micha{\l} and Mizerka, Piotr},
  title        = {Replication details for "Sampling elements of a finite group: efficiency of the Product Replacement Algorithm with accumulator"},
  year         = {2025},
  doi          = {10.5281/zenodo.17981235},
  url          = {https://doi.org/10.5281/zenodo.17981235},
  publisher    = {Zenodo}
}

@article{Lubin2023JuMP,
  author  = {Lubin, Miles and Dowson, Oscar and Garcia, Joaquin Diaz and Huchette, Joey and Legat, Beno{\^\i}t and Vielma, Juan Pablo},
  title   = {JuMP 1.0: Recent Improvements to a Modeling Language for Mathematical Optimization},
  journal = {Mathematical Programming Computation},
  year    = {2023},
  doi     = {10.1007/s12532-023-00235-5}
}

@article{ODonoghue2016Conic,
  author  = {O'Donoghue, Brendan and Chu, Eric and Parikh, Neal and Boyd, Stephen},
  title   = {Conic Optimization via Operator Splitting and Homogeneous Self-Dual Embedding},
  journal = {Journal of Optimization Theory and Applications},
  volume  = {169},
  number  = {3},
  pages   = {1042--1068},
  year    = {2016},
  month   = {June},
  doi     = {10.1007/s10957-016-0892-3}
}

@article{Julia-2017,
    title={Julia: A fresh approach to numerical computing},
    author={Bezanson, Jeff and Edelman, Alan and Karpinski, Stefan and Shah, Viral B},
    journal={SIAM {R}eview},
    volume={59},
    number={1},
    pages={65--98},
    year={2017},
    publisher={SIAM},
    doi={10.1137/141000671},
    url={https://epubs.siam.org/doi/10.1137/141000671}
}

@misc{mizerka_pra_complexity_2025,
  author       = {Marcinkowski, Michał and Mizerka, Piotr},
  title        = {pra\_complexity},
  year         = {2025},
  howpublished = {\url{https://github.com/piotrmizerka/pra_complexity}},
  note         = {GitHub repository},
}

@misc{kaluba_PropertyTjl_2024,
  author       = {Kaluba, Marek},
  title        = {{PropertyT.jl}: Sum-of-squares methods in group rings for certifying Kazhdan Property (T)},
  year         = {2024},
  howpublished = {\url{https://github.com/kalmarek/PropertyT.jl}},
}

@book {MR2466937,
    AUTHOR = {Levin, David A. and Peres, Yuval and Wilmer, Elizabeth L.},
     TITLE = {Markov chains and mixing times},
      NOTE = {With a chapter by James G. Propp and David B. Wilson},
 PUBLISHER = {American Mathematical Society, Providence, RI},
      YEAR = {2009},
     PAGES = {xviii+371},
      ISBN = {978-0-8218-4739-8},
   MRCLASS = {60J10 (60-01 60J05 60K35 60K37 68U20 68W20)},
  MRNUMBER = {2466937},
MRREVIEWER = {Olle\ H\"aggstr\"om},
       DOI = {10.1090/mbk/058},
       URL = {https://doi.org/10.1090/mbk/058},
}

@article {MR1356111,
    AUTHOR = {Celler, Frank and Leedham-Green, Charles R. and Murray, Scott
              H. and Niemeyer, Alice C. and O'Brien, E. A.},
     TITLE = {Generating random elements of a finite group},
   JOURNAL = {Comm. Algebra},
  FJOURNAL = {Communications in Algebra},
    VOLUME = {23},
      YEAR = {1995},
    NUMBER = {13},
     PAGES = {4931--4948},
      ISSN = {0092-7872,1532-4125},
   MRCLASS = {20P05 (20D99 68Q20)},
  MRNUMBER = {1356111},
MRREVIEWER = {P.\ P.\ P\'alfy},
       DOI = {10.1080/00927879508825509},
       URL = {https://doi.org/10.1080/00927879508825509},
}

@article {PakZuk,
    AUTHOR = {Pak, Igor and \.{Z}uk, Andrzej},
     TITLE = {On {K}azhdan constants and mixing of random walks},
   JOURNAL = {Int. Math. Res. Not.},
  FJOURNAL = {International Mathematics Research Notices},
      YEAR = {2002},
    NUMBER = {36},
     PAGES = {1891--1905},
      ISSN = {1073-7928,1687-0247},
   MRCLASS = {20F65 (20P05 60G50)},
  MRNUMBER = {1920168},
MRREVIEWER = {Alexander\ Gamburd},
       DOI = {10.1155/S1073792802206017},
       URL = {https://doi.org/10.1155/S1073792802206017},
}

@article {MR4912919,
    AUTHOR = {Kaluba, Marek and Kielak, Dawid},
     TITLE = {Kazhdan constants for {C}hevalley groups over the integers},
   JOURNAL = {Rev. Mat. Iberoam.},
  FJOURNAL = {Revista Matem\'atica Iberoamericana},
    VOLUME = {41},
      YEAR = {2025},
    NUMBER = {4},
     PAGES = {1253--1269},
      ISSN = {0213-2230,2235-0616},
   MRCLASS = {22D55},
  MRNUMBER = {4912919},
       DOI = {10.4171/rmi/1534},
       URL = {https://doi.org/10.4171/rmi/1534},
}

@misc{nitsche2022computerproofspropertyt,
      title={Computer proofs for {P}roperty ({T}), and {SDP} duality}, 
      author={Martin Nitsche},
      year={2022},
      eprint={2009.05134},
      archivePrefix={arXiv},
      primaryClass={math.GR},
      howpublished={\url{https://arxiv.org/abs/2009.05134}}, 
}

@article {MR3427595,
    AUTHOR = {Ozawa, Narutaka},
     TITLE = {Noncommutative real algebraic geometry of {K}azhdan's property
              ({T})},
   JOURNAL = {J. Inst. Math. Jussieu},
  FJOURNAL = {Journal of the Institute of Mathematics of Jussieu. JIMJ.
              Journal de l'Institut de Math\'ematiques de Jussieu},
    VOLUME = {15},
      YEAR = {2016},
    NUMBER = {1},
     PAGES = {85--90},
      ISSN = {1474-7480,1475-3030},
   MRCLASS = {22D10 (20F10 22D15 46L89)},
  MRNUMBER = {3427595},
MRREVIEWER = {Alain\ Valette},
       DOI = {10.1017/S1474748014000309},
       URL = {https://doi.org/10.1017/S1474748014000309},
}

@inproceedings{10.1145/96877.96918,
author = {Cooperman, G. and Finkelstein, L. and Sarawagi, N.},
title = {A random base change algorithm for permutation groups},
year = {1990},
isbn = {0201548925},
publisher = {Association for Computing Machinery},
address = {New York, NY, USA},
url = {https://doi.org/10.1145/96877.96918},
doi = {10.1145/96877.96918},
pages = {161–168},
numpages = {8},
location = {Tokyo, Japan},
series = {ISSAC '90}
}

@article{10.1112/plms/s3-65.3.555,
    author = {Neumann, Peter M. and Praeger, Cheryl E.},
    title = {A Recognition Algorithm for Special Linear Groups},
    journal = {Proceedings of the London Mathematical Society},
    volume = {s3-65},
    number = {3},
    pages = {555-603},
    year = {1992},
    month = {11},
    issn = {0024-6115},
    doi = {10.1112/plms/s3-65.3.555},
    url = {https://doi.org/10.1112/plms/s3-65.3.555},
    eprint = {https://academic.oup.com/plms/article-pdf/s3-65/3/555/4456828/s3-65-3-555.pdf},
}

@inproceedings{10.1145/103418.103440,
author = {Babai, L\'{a}szl\'{o}},
title = {Local expansion of vertex-transitive graphs and random generation in finite groups},
year = {1991},
isbn = {0897913973},
publisher = {Association for Computing Machinery},
address = {New York, NY, USA},
url = {https://doi.org/10.1145/103418.103440},
doi = {10.1145/103418.103440},
booktitle = {Proceedings of the Twenty-Third Annual ACM Symposium on Theory of Computing},
pages = {164–174},
numpages = {11},
location = {New Orleans, Louisiana, USA},
series = {STOC '91}
}

@article{Dixon2008,
author = {Dixon, John D.},
journal = {The Electronic Journal of Combinatorics [electronic only]},
language = {eng},
number = {1},
pages = {Research Paper R94, 13 p.-Research Paper R94, 13 p.},
publisher = {Prof. André Kündgen, Deptartment of Mathematics, California State University San Marcos, San Marcos},
title = {Generating random elements in finite groups.},
url = {http://eudml.org/doc/130477},
volume = {15},
year = {2008},
}

@article {MR1112272,
    AUTHOR = {Knuth, Donald E.},
     TITLE = {Efficient representation of perm groups},
   JOURNAL = {Combinatorica},
  FJOURNAL = {Combinatorica. An International Journal on Combinatorics and
              the Theory of Computing},
    VOLUME = {11},
      YEAR = {1991},
    NUMBER = {1},
     PAGES = {33--43},
      ISSN = {0209-9683},
   MRCLASS = {20B40 (20-04 68Q25)},
  MRNUMBER = {1112272},
MRREVIEWER = {W.\ M.\ Kantor},
       DOI = {10.1007/BF01375471},
       URL = {https://doi.org/10.1007/BF01375471},
}

@book{Seress_2003, 
place={Cambridge}, 
series={Cambridge Tracts in Mathematics}, title={Permutation Group Algorithms}, publisher={Cambridge University Press}, author={Seress, {\'A}kos}, year={2003}, collection={Cambridge Tracts in Mathematics}
}

@software{Sanders2023IntervalArithmetic,
  author  = {Sanders, D. P. and Benet, L. and Ferranti, L. and Agarwal, K. and Richard, B.
             and Grawitter, J. and Gupta, E. and Forets, M. and Herbst, M. F.
             and yashrajgupta and Hanson, E. and van Dyk, B. and Rackauckas, C.
             and Vasani, R. and Miclu{\c{t}}a-C{\^a}mpeanu, S. and Olver, S.
             and Koolen, T. and Wormell, C. and Karrasch, D. and Widmann, D.
             and V{\'a}zquez, F. A. and Dalle, G. and Sarnoff, J. and TagBot, J.
             and O'Bryant, K. and Carlsson, K. and Piibeleht, M. and Giordano, M.
             and Zhang, Q. and Deits, R.},
  title   = {{JuliaIntervals/IntervalArithmetic.jl}: v0.20.9},
  year    = {2023},
  publisher = {Zenodo},
  version = {0.20.9},
  doi     = {10.5281/zenodo.8037734},
  url     = {https://doi.org/10.5281/zenodo.8037734}
}

@article {MR1815215,
    AUTHOR = {Lubotzky, Alexander and Pak, Igor},
     TITLE = {The product replacement algorithm and {K}azhdan's property
              ({T})},
   JOURNAL = {J. Amer. Math. Soc.},
  FJOURNAL = {Journal of the American Mathematical Society},
    VOLUME = {14},
      YEAR = {2001},
    NUMBER = {2},
     PAGES = {347--363},
      ISSN = {0894-0347,1088-6834},
   MRCLASS = {60B15 (05C25 22D10 60J10)},
  MRNUMBER = {1815215},
MRREVIEWER = {Piotr\ Graczyk},
       DOI = {10.1090/S0894-0347-00-00356-8},
       URL = {https://doi.org/10.1090/S0894-0347-00-00356-8},
}

@article {MR4224715,
    AUTHOR = {Kaluba, Marek and Kielak, Dawid and Nowak, Piotr W.},
     TITLE = {On property ({T}) for {$\operatorname{Aut}(F_n)$} and {$\operatorname{SL}_n(\Bbb{Z})$}},
   JOURNAL = {Ann. of Math. (2)},
  FJOURNAL = {Annals of Mathematics. Second Series},
    VOLUME = {193},
      YEAR = {2021},
    NUMBER = {2},
     PAGES = {539--562},
      ISSN = {0003-486X,1939-8980},
   MRCLASS = {22D55 (20F28)},
  MRNUMBER = {4224715},
MRREVIEWER = {Marco\ Trombetti},
       DOI = {10.4007/annals.2021.193.2.3},
       URL = {https://doi.org/10.4007/annals.2021.193.2.3},
}

@article {MR3794918,
    AUTHOR = {Kaluba, Marek and Nowak, Piotr W.},
     TITLE = {Certifying numerical estimates of spectral gaps},
   JOURNAL = {Groups Complex. Cryptol.},
  FJOURNAL = {Groups. Complexity. Cryptology},
    VOLUME = {10},
      YEAR = {2018},
    NUMBER = {1},
     PAGES = {33--41},
      ISSN = {1867-1144,1869-6104},
   MRCLASS = {22D10 (20C07 20C40 94A60)},
  MRNUMBER = {3794918},
       DOI = {10.1515/gcc-2018-0004},
       URL = {https://doi.org/10.1515/gcc-2018-0004},
}

@article {MR3359223,
    AUTHOR = {Netzer, Tim and Thom, Andreas},
     TITLE = {Kazhdan's property ({T}) via semidefinite optimization},
   JOURNAL = {Exp. Math.},
  FJOURNAL = {Experimental Mathematics},
    VOLUME = {24},
      YEAR = {2015},
    NUMBER = {3},
     PAGES = {371--374},
      ISSN = {1058-6458,1944-950X},
   MRCLASS = {16S34 (20C07 22D10)},
  MRNUMBER = {3359223},
       DOI = {10.1080/10586458.2014.999149},
       URL = {https://doi.org/10.1080/10586458.2014.999149},
}

@article {MR4709065,
    AUTHOR = {Leedham-Green, C. R.},
     TITLE = {On a variant of the product replacement algorithm},
   JOURNAL = {Glasg. Math. J.},
  FJOURNAL = {Glasgow Mathematical Journal},
    VOLUME = {66},
      YEAR = {2024},
    NUMBER = {1},
     PAGES = {221--228},
      ISSN = {0017-0895,1469-509X},
   MRCLASS = {20P05 (62M05)},
  MRNUMBER = {4709065},
       DOI = {10.1017/s0017089523000435},
       URL = {https://doi.org/10.1017/s0017089523000435},
}

@article {MR4023374,
    AUTHOR = {Kaluba, Marek and Nowak, Piotr W. and Ozawa, Narutaka},
     TITLE = {{${\rm Aut}(\Bbb F_5)$} has property {$(T)$}},
   JOURNAL = {Math. Ann.},
  FJOURNAL = {Mathematische Annalen},
    VOLUME = {375},
      YEAR = {2019},
    NUMBER = {3-4},
     PAGES = {1169--1191},
      ISSN = {0025-5831,1432-1807},
   MRCLASS = {20F28 (20C15)},
  MRNUMBER = {4023374},
MRREVIEWER = {Fausto\ De Mari},
       DOI = {10.1007/s00208-019-01874-9},
       URL = {https://doi.org/10.1007/s00208-019-01874-9},
}

@incollection {MR1829489,
    AUTHOR = {Pak, Igor},
     TITLE = {What do we know about the product replacement algorithm?},
 BOOKTITLE = {Groups and computation, {III} ({C}olumbus, {OH}, 1999)},
    SERIES = {Ohio State Univ. Math. Res. Inst. Publ.},
    VOLUME = {8},
     PAGES = {301--347},
 PUBLISHER = {de Gruyter, Berlin},
      YEAR = {2001},
      ISBN = {3-11-016721-2},
   MRCLASS = {20P05 (20F69 68W30)},
  MRNUMBER = {1829489},
MRREVIEWER = {Tatiana\ Smirnova-Nagnibeda},
}

@incollection {MR2053017,
    AUTHOR = {Babai, L\'aszl\'o{} and Pak, Igor},
     TITLE = {Strong bias of group generators: an obstacle to the ``product
              replacement algorithm''},
      NOTE = {SODA 2000 special issue},
   JOURNAL = {J. Algorithms},
  FJOURNAL = {Journal of Algorithms. Cognition, Informatics and Logic},
    VOLUME = {50},
      YEAR = {2004},
    NUMBER = {2},
     PAGES = {215--231},
      ISSN = {0196-6774},
   MRCLASS = {60B15 (20F99)},
  MRNUMBER = {2053017},
MRREVIEWER = {C.\ R. E. Raja},
       DOI = {10.1016/S0196-6774(03)00091-9},
       URL = {https://doi.org/10.1016/S0196-6774(03)00091-9},
}

@incollection {MR1929718,
    AUTHOR = {Leedham-Green, C. R. and Murray, Scott H.},
     TITLE = {Variants of product replacement},
 BOOKTITLE = {Computational and statistical group theory ({L}as {V}egas,
              {NV}/{H}oboken, {NJ}, 2001)},
    SERIES = {Contemp. Math.},
    VOLUME = {298},
     PAGES = {97--104},
 PUBLISHER = {Amer. Math. Soc., Providence, RI},
      YEAR = {2002},
      ISBN = {0-8218-3158-5},
   MRCLASS = {20-04 (60J10 68W20 68W30)},
  MRNUMBER = {1929718},
MRREVIEWER = {Burkhard\ K.\ H\"ofling},
       DOI = {10.1090/conm/298/05116},
       URL = {https://doi.org/10.1090/conm/298/05116},
}

\end{document}